\documentclass[letterpaper, 10 pt, conference]{ieeeconf}  

\IEEEoverridecommandlockouts                              

\usepackage{balance}

\usepackage[english]{babel}
\usepackage{amssymb,amsmath} 
\usepackage[latin1,utf8]{inputenc}
\usepackage{graphicx}
\usepackage{epstopdf}
\usepackage{comment}
\usepackage{subcaption}
\usepackage{wrapfig}
\usepackage{cite}
\usepackage[utf8]{inputenc}
\usepackage[mathscr]{euscript}
\usepackage{hyperref}
\usepackage{color}
\usepackage{bm}

\usepackage[english]{babel}
\usepackage[utf8]{inputenc}
\usepackage{graphicx}
\usepackage{listings}
\usepackage{verbatim}
\usepackage{bigints}
\usepackage{comment}
\usepackage{braket}
\usepackage{epigraph}
\usepackage{float}
\usepackage{mathtools}
\usepackage{pgfplots}
\usepackage{algpseudocode}
\usepackage{algorithm}
\usepackage{tikz-network}
\usepackage{tikz}
\usepackage{comment}
\usepackage{hyperref}
\hypersetup{colorlinks, breaklinks, linkcolor=black,citecolor=black,filecolor=black,urlcolor=black} 
\usepackage{graphicx} 
\usepackage{color} 
\usepackage{amsmath,amssymb, amsfonts} 
\usepackage{listings} 
\usepackage{booktabs} 
\usepackage{xspace} 
\usepackage{tabularx}
\usepackage{xcolor}
\usepackage{gensymb}
\usepackage{textcomp}
\usepackage[printonlyused,withpage]{acronym} 
\usepackage{rotating} 
\usepackage{microtype} 
\usepackage{mathrsfs}
\usepackage{cleveref}
\usepackage{tikz-cd}
\usepackage{wrapfig}
\usepackage{dsfont}
\usepackage{lipsum}
\usepackage{longtable}
\usepackage{siunitx}
\usepackage{subfiles}
\graphicspath{{images/}{../images/}}
\usepackage{graphicx}
\usepackage{subcaption}
\usepackage{float}
\usepackage{dsfont}
\usepackage{physics}
\usepackage[normalem]{ulem}
\usepackage{soul}
\usepackage{tikz}
\usetikzlibrary{fit,positioning,arrows,automata}

\DeclareMathOperator{\dom}{dom}

\def\mR{{\mathbb R}}

\DeclareMathOperator{\diag}{diag} 
\DeclareMathOperator{\Span}{span} 

\newcommand{\Tau}{\mathcal{T}}

\def\XXint#1#2#3{{\setbox0=\hbox{$#1{#2#3}{\int}$} 
		\vcenter{\hbox{$#2#3$}}\kern-.5\wd0}}

\renewcommand{\hat}[1]{\widehat{#1}}
\renewcommand{\theta}{\vartheta}
\renewcommand{\epsilon}{\varepsilon}

\def\minwrt[#1]{\underset{#1}{\text{minimize }}}
\def\argminwrt[#1]{\underset{#1}{\text{arg min }}}
\def\maxwrt[#1]{\underset{#1}{\text{maximize }}}
\def\argmaxwrt[#1]{\underset{#1}{\text{arg max }}}
\def\maxemphwrt[#1]{\underset{#1}{\text{\emph{maximize} }}}

\def\minwrt[#1]{\underset{#1}{\text{minimize }}}
\def\argminwrt[#1]{\underset{#1}{\text{arg min }}}
\def\maxwrt[#1]{\underset{#1}{\text{maximize }}}
\def\argmaxwrt[#1]{\underset{#1}{\text{arg max }}}
\def\maxemphwrt[#1]{\underset{#1}{\text{\emph{maximize} }}}
\newcommand{\ett}{{\bf 1}}

\newtheorem{remark}{Remark}
\newtheorem{proposition}{Proposition}
\newtheorem{corollary}{Corollary}
\newtheorem{lemma}{Lemma}
\newtheorem{definition}{Definition}

\newtheorem{assumption}{Assumption}

\crefname{assumption}{Assumption}{Assumptions}
\Crefname{assumption}{Assumption}{Assumptions}

\def\ccD{{\mathcal{D}}}

\def\RN{{\mathbb{N}}}
\def\RR{{\mathbb{R}}}

\definecolor{Michele}{HTML}{009B55}

\DeclareMathOperator{\suto}{subject\;to}
\DeclareMathOperator*{\argmin}{arg\,min}

\newcommand{\T}{\Tau}
\usepackage{subfloat}
\usepackage{stfloats}
\usepackage{comment}

\title{\LARGE \bf
A convex approach for Markov chain estimation\\ from aggregate data via inverse optimal transport}

\author{Michele Mascherpa,  Axel Ringh, Amirhossein Taghvaei, and Johan Karlsson
\thanks{This work was partially supported by KTH Digital Futures,
by the Wallenberg AI, Autonomous Systems and Software Program (WASP) funded by the Knut and Alice Wallenberg Foundation, Sweden, the Swedish Research Council (VR) under grant 2020-03454, 
and by National Science Foundation (NSF) awards EPCN-2318977, EPCN-2347358.}
\thanks{M.~Mascherpa and J.~Karlsson are with the Department of Mathematics and Digital Futures, KTH Royal Institute of Technology, Stockholm. {\tt\small micmas@kth.se}, {\tt\small johan.karlsson@math.kth.se}.
A.~Ringh  is with the Department of Mathematical Sciences,  Chalmers University of Technology and University of Gothenburg, Gothenburg, Sweden. {\tt\small axelri@chalmers.se}.
Amirhossein Taghvaei is with William E. Boeing department of Aeronautics and Astronautics at the University of Washington (UW) Seattle, {\tt\small amirtag@uw.edu}.}
}

\begin{document}


\maketitle
\thispagestyle{empty}
\pagestyle{empty}


\begin{abstract}
We address the problem of identifying the dynamical law governing the evolution of a population of indistinguishable particles, when only aggregate distributions at successive times are observed. Assuming a Markovian evolution on a discrete state space, the task reduces to estimating the underlying transition probability matrix from distributional data. We formulate this inverse problem within the framework of entropic optimal transport, as a joint optimization over the transition matrix and the transport plans connecting successive distributions. This formulation results in a convex optimization problem, and we propose an efficient iterative algorithm based on the entropic proximal method. We illustrate the accuracy and convergence of the method in two numerical setups, considering estimation from independent snapshots and estimation from a time series of aggregate observations, respectively.

\end{abstract}

\section{Introduction}
Consider a scenario where a large number of indistinguishable particles move according to an unknown dynamical law. Although the individual trajectories of the particles are not accessible, one can observe their aggregate configurations at successive time instances. The fundamental challenge is to identify the underlying—potentially stochastic—law that governs the evolution of these particles using only such aggregate data. This inverse problem arises naturally in diverse domains. In fluid mechanics, for example, Particle Tracking Velocimetry (PTV) techniques rely on the movement of tracer particles illuminated by lasers to infer local flow dynamics \cite{dabiri2019particle}. Beyond fluid flow estimation, similar challenges appear in many other applications. For example, in orbital object tracking, where space debris or satellites follow uncertain dynamical laws \cite{anz2020orbital,berger2020flying}, and in swarm robotics, where collective motion patterns of large groups of simple agents must be analyzed and predicted \cite{rubenstein2012kilobot}. Further applications include understanding pedestrian movement from anonymized cell phone data \cite{willenborg2025using} or inferring individual-level animal behavior from aggregate ecology data \cite{king2013solution}.


Given the indistinguishable nature of the observed particles, it is natural to represent each snapshot as a distribution over the state space. When this space is discrete, the aggregate evolution can be modeled as a sequence of distributions generated by a Markov chain. In this setting, the underlying stochastic law corresponds to the transition probability matrix that governs how mass is transported between states over time. The key scientific problem, therefore, is to estimate this transition matrix from a sequence of noisy or incomplete distributional measurements.

The estimation of transition probabilities from Markov models is a classical topic \cite{miller1952finite, kemeny1969finite}, and in particular the estimation from aggregate discrete-time observations has been considered  in the monograph \cite{lee1970estimating}, which derives least-squares and maximum-likelihood estimators formulated as quadratic programming problems.
For a continuous-time Markov process, the same type of estimators have been studied in \cite{kalbfleisch1983estimation}, whose results are consolidated in~\cite{lawless1984information}, analyzing the amount of information contained in aggregate data and showing the efficiency loss that arises when individual transitions are unobserved. 
In \cite{bernstein2016consistently} the scenario in which data are corrupted by noise is considered, and an estimator based on the method of moments is proposed.

While previous methods are mainly based on least squares or covariance/moment based estimates, we here propose a new likelihood-based method using concepts from optimal transport. The optimal transport problem, which is a classical problem \cite{villani2021topics}, has recently been used to address problems in estimation and control \cite{chen2021optimal, haasler2021control, mascherpa2023estimating, terpin2024dynamic, emerick2024causal}. 
In \cite{stuart2020inverse} the concept of inverse optimal transport was introduced, in which, given the marginals, the nonconvex problem of simultaneously finding both the cost and the transport matrices is considered. 
This concept has been explored in various settings, often motivated by applications in biology.
In \cite{lamoline2025dynamic} and \cite{elvander2025mixtures}, inverse optimal transport problems are used to identify the dynamics given sampling data, and nonconvex problems  of jointly estimating dynamics and transport plans are considered. 
Other related works include  \cite{swaminathan2014identification}, where Markov-chain identification from distributional data is posed as maximum a posteriori estimation with a multinomial likelihood under dynamic constraints, yielding a non-convex program.
Related work~\cite{yang2023estimating} also considers inverse problems from aggregate marginal observations, but focuses on recovering time-varying optimal transport cost matrices. In contrast, we estimate a single Markov transition matrix shared across all observations, which enables us to analyze identifiability and uniqueness of the recovered transition matrix. Other related approaches use optimal transport or Wasserstein-gradient-flow models to infer state dynamics from population snapshots~\cite{schiebinger2019optimal,terpin2024learning,
lavenant2024toward, tronstad2025multistageot}. While related in their use of marginal snapshot data, their aim is to recover couplings, path laws, or energy functionals, rather than a finite-state transition matrix.

Our approach can be seen as a regularized version of inverse optimal transport, where transport plans and transition probabilities are estimated. The resulting optimization problem is jointly convex and admits an efficient proximal algorithm that updates the transport matrices and normalizes their aggregate to recover the transition probabilities. We derive the corresponding dual problem, which we use to provide uniqueness conditions as well as convergence of the method.

The article is organized as follows. In \Cref{sec:background}, we provide background on the Schrödinger bridge problem and its connection to entropic optimal transport. In \Cref{sec:problem}, we formulate the joint optimization problem over transport matrices and the transition probability matrix, and analyze existence, uniqueness, and the corresponding dual formulation. In \Cref{sec:nm}, we introduce an iterative algorithm for efficiently solving the problem, while \Cref{sec:application} discusses identifiability issues and presents numerical results under varying data dependencies and noise levels, including a comparison with a least-squares estimator. The paper concludes with a summary of findings in \Cref{sec:conc}.

\section{Background}
\label{sec:background}
\subsection{Notation}
We denote by $\ett_n$ a vector of ones, omitting the dimension $n$ if clear from context. With $\odot$ we indicate entrywise matrix multiplication, and we use $\langle \cdot,\cdot\rangle$ to denote the standard (vector/Frobenius) inner product.
The nonnegative orthant of $\mR^n$ is denoted with $\mR^n_+$, and $[m]:=\{1,\ldots,m\}$.

\subsection{Log-likelihoods for collections of particles in a Markov chain and connection to optimal transport} \label{subsect:schrodinger}

Consider a collection of indistinguishable particles, where each particle evolves over a finite set of states $X= \left\{ X_1, X_2, \dots, X_n \right\}$ \cite{haasler2019estimating, pavon2010discrete}. Assume that the underlying probability law is specified by a state transition matrix $A=(a_{ij})_{i,j = 1}^n$, where $a_{ij} = P( q_{t+1}=X_j | q_t = X_i )$, and $q_t$ denotes the state of a particle at time $t$.
Now consider the case where we observe the particle distributions $\mu_0,\mu_1\in \RN^{n}$ before and after a transition, 
where the $i$-th component $\mu_t(i)$ denotes the number of particles in state $X_i$ at time $t\in \{0,1\}$.
The particle transitions between states can be described by a matrix $M=(m_{ij})_{i,j=1}^n$, where $m_{ij}$ denotes the number of particles that move from state $X_i$ to state $X_j$.
By construction, the mass transition matrix satisfies $\sum_i m_{ij} = \mu_1(j)$ for all $j\in [n]$ and  $\sum_j m_{ij}=\mu_0(i)$ for all $i\in [n]$.

Given the initial distribution $\mu_0$ and the transition probability matrix $A$, the probability that the particles evolve according to a given transition matrix $M$ is 
\begin{equation} \label{eq:probabilityM}
P_{\mu_0,A}( M ) = \prod_{i=1}^{n} \left( \binom{\mu_0(i)}{m_{i1},m_{i2},\dots,m_{in}} \prod_{j=1}^{n} a_{ij}^{m_{ij}} \right),
\end{equation}
where $\binom{\cdot}{\cdot, \ldots, \cdot}$ denotes a multinomial coefficient.
When the number of particles grows, the probability of a given transition can be approximated by the relative entropy \cite{haasler2019estimating}
\[
P_{\mu_0,\, A}(M)\approx e^{-\ccD \left( M \, \big| \, \diag(\mu_0)A  \right)},
\]
where $\ccD$ is the Kullback-Leibler (KL) divergence.
\begin{definition}
Let $p$ and $q$ be two nonnegative vectors or matrices of the same dimension. The Kullback-Leibler (KL)  functional of $p$ from $q$ is defined as 
\begin{equation}
\label{eq:kl}
    \ccD(p|q):=\sum_i p_i \log \frac{p_i}{q_i},
\end{equation}
where $0\log 0$ is defined to be $0$. When $p$ and $q$ have the same total mass, i.e., $\sum_i p_i=\sum_i q_i$, this functional is a divergence and is nonnegative.
 In this work, we also use $\ccD$ with arguments that may have different total mass, in which case it can be negative.
\end{definition}
This formulation allows us to both approximate the negative log-likelihood of \eqref{eq:probabilityM} in terms of the KL divergence and to consider particles as continuous quantities, that is, $M\in \RR_+^{n\times n}$;  as the number of particles becomes large, the discrete particle distributions can be approximated by continuous densities.
A feasible evolution matrix between the two distributions $\mu_0\in\RR_+^{n}$ and $\mu_1\in\RR_+^{n}$ must satisfy the marginal constraints $M \ett=\mu_0$ and $M^T \ett=\mu_1$.
Hence, the most likely evolution matrix $M$ between the two distributions can be approximated by the solution to
\begin{equation} \label{eq:sb_1T}
\begin{aligned}
	\min_{M \in \mR^{n\times n}_+} \; \ &  \ccD( M \,|\,\diag(\mu_{0})A) \\
	\text{subject to } \  &  M \ett = \mu_{0}  , \ \  M^T \ett = \mu_{1} .
 \end{aligned}
\end{equation}
This problem can naturally be extended to a Markov chain of length $\T$, and can be referred to as a time- and space-discrete form of the Schrödinger bridge problem \cite{pavon2010discrete, haasler2019estimating, haasler2021tree, mascherpa2023estimating}.

It should be noted that the formulation \eqref{eq:sb_1T} is closely connected to the (entropy regularized) optimal transport problem. In fact, if $A$ and $\mu_0$ are strictly positive, then \eqref{eq:sb_1T} is equivalent to the problem
\begin{equation} \label{eq:sb_1T11}
\begin{aligned}
	\min_{M \in \mR^{n\times n}_+} \; \ &  \langle C,M\rangle+\ccD( M \,|\ett_{n\times n}) \\
	\text{subject to } \  &  M \ett = \mu_{0}  , \ \  M^T \ett = \mu_{1},
 \end{aligned}
\end{equation}
where the cost matrix is $C=-\log(\diag(\mu_0)A)$.  %
A key difference between the two formulations is that $\ccD( M |\diag(\mu_{0})A)$ is jointly convex in $M$ and $A$, whereas $\langle C,M\rangle$ is not jointly convex in $M$ and $C$. This distinction is consequential in the present setting because $A$ (or equivalently $C$) is unknown,  and our aim is to identify the transition matrix and the transport matrix simultaneously. This is the main reason for adopting the former formulation here.

\section{Problem formulation, duality, existence and uniqueness}
\label{sec:problem}

For a sequence of observed marginal pairs $(\mu_t,\nu_t)$, with $t=1,\ldots,\T$, satisfying $\mu_t^T \ett = \nu_t^T \ett$ for all $t$, we want to simultaneously find the $\T$ transport matrices $M_t \in \mR^{n \times n}_+$ connecting each initial distribution $\mu_t \in \mR^n_+$ to its corresponding final distribution $\nu_t \in \mR^n_+$, and their common prior $A \in \mR^{n \times n}_+$, which is the most likely discrete Markov chain, in the Kullback-Leibler divergence sense, to have generated the marginal observations. Formulated as an optimization problem, this means that we want to solve
\begin{subequations}
\label{eq:primalfirstform}
\begin{align}
    \min_{\substack{A, M_t\in \mR^{n\times n}_+\\ t \in [\T]}} \quad &\sum_{t=1}^\T \ccD(M_t|\diag(\mu_t)A)\\
    \suto\quad & A\ett=\ett \label{eq:normal}\\
                    & M_t\ett =\mu_t \qquad \;\mbox{for } t\in[\T] \label{eq:firstconst}\\
                        &M_t^T\ett=\nu_t \qquad  \mbox{for } t\in[\T].\label{eq:secondconst}
\end{align}
\end{subequations}

Constraint \eqref{eq:normal} ensures that the matrix $A$ is row-stochastic, and thus its entries can be interpreted as transition probabilities. Conditions \eqref{eq:firstconst} and \eqref{eq:secondconst} are marginal constraints on the mass transport matrices $M_t$.

Note that, under the above constraints, we have
\begin{align*}
&\ccD(M_t|\diag(\mu_t)A)=\sum_{i,j}M_t(i,j)\log \left(\frac{M_t(i,j)}{\mu_t(i) {A}(i,j)}\right)\\
&=\sum_{i,j}M_t(i,j)\log \left(\frac{M_t(i,j)}{ A(i,j)}\right)-\sum_{i}\mu_t(i)\log \mu_t(i).
\end{align*}
The last term is constant and thus does not affect the optimization. 
The problem can thus be formulated as
\begin{equation}
\label{eq:problem2}
\begin{aligned}
     \min_{\substack{A, M_t\in \mR^{n\times n}_+\\ t \in [\T]} }\quad &\sum_{t=1}^\T \ccD(M_t|A)\\
    \suto\quad & A\ett=\ett\\
                     & M_t\ett =\mu_t && \mbox{for } t \in [\T] \\
                        &M_t^T\ett=\nu_t && \mbox{for } t\in[\T].
\end{aligned}
\end{equation}
We first consider the minimization over $A$, with the transport matrices $M_t$ fixed. 
Then, the part of the objective function of \eqref{eq:problem2} depending on $A$ is $\ccD( \bar M\mid  A)$, where $\bar {M}(i,j) := \sum_{t=1}^\T M_t(i,j)$. 
Since the constraint $A\ett=\ett$ acts row-wise, the minimization separates over the rows of $A$.  Using standard properties of the Kullback-Leibler divergence, $A$ is the solution where each row is proportional to the corresponding row of $\bar {M}$, i.e., 
\begin{equation}
\label{eq:optimalA}
    A(i,j)=
    \frac{\bar M(i,j)}{\sum_\ell \bar M(i,\ell)}.
\end{equation}
Plugging \eqref{eq:optimalA} into \eqref{eq:problem2}, and removing constant terms in the objective function, we obtain the equivalent formulation
\begin{equation}
\label{eq:newformulation}
\begin{aligned}
    \min_{\substack{\bar M, M_t\in \mR^{n\times n}_+\\ t \in [\T]} } \quad &\sum_{t=1}^\T \ccD(M_t|\bar{M})\\
    \suto\quad  & M_t\ett =\mu_t && \mbox{for }  t\in [\T] \\
                        &M_t^T\ett=\nu_t && \mbox{for }  t\in [\T] \\
                        & \bar {M} = \sum_{t=1}^\T M_t,
\end{aligned}
\end{equation}
where $A$ is removed from the problem and can be recovered from $\bar M$ via \eqref{eq:optimalA}.
Since \eqref{eq:primalfirstform} and \eqref{eq:newformulation} are equivalent, and in particular yield the same optimal transport plans $(M_t)_{t=1}^\T$, we will study the latter formulation.

\subsection{Existence of an optimal solution}
We make the following assumption on the observed marginals $\mu_t,\nu_t$.
\begin{assumption}\label{as:nonzero}
     For all index pairs $(i,j)\in [n]\times [n]$, there exists $t \in [\T]$ such that $\mu_t(i)>0$ and $\nu_t(j)>0$.
\end{assumption}

\begin{remark}
If $\mu_t(i)=0$ or $\nu_t(j)=0$, then the only feasible value is $M_t(i,j)=0$. In particular, if \Cref{as:nonzero} does not hold, i.e., if there is an index pair $(i,j) \in [n] \times [n]$ such that for every $t\in[\T]$ either $\mu_t(i)=0$ or $\nu_t(j)=0$, then we have that $M_t(i,j)=0$ for all $t$, implying $\bar M(i,j)=0$, and the problem can be restricted to the remaining variables.
\end{remark}

\begin{lemma}
\label{lemma:feasiblity}
There exists a feasible solution to \eqref{eq:newformulation}. Moreover, under \Cref{as:nonzero}, there exists a feasible solution such that $\bar M>0$.
\end{lemma}
\begin{proof}
    Let $s_t:=\mu_t^T\ett=\nu_t^T\ett$. Consider $M_t:=\frac{1}{s_t}\mu_t\nu_t^T$, for all $t\in[\T]$. Then $M_t\geq0, M_t \ett=\mu_t$ and $M_t^T \ett =\nu_t$ for all $t$, hence the problem is feasible. Furthermore, under \Cref{as:nonzero}, for every $(i,j) \in [n]\times[n]$ there exists $t \in [\T]$ such that $\mu_t(i)>0$ and $\nu_t(j)>0$. Therefore we also have $\bar M = \sum_{t=1}^\T M_t = \sum_{t=1}^\T \frac{1}{s_t}\mu_t\nu_t^T>0$.
\end{proof}
The next result is that the set of optimal solutions of \eqref{eq:newformulation} is non-empty.
\begin{proposition}
\label{prop:existence}
    There exists an optimal solution to \eqref{eq:newformulation}. 
\end{proposition}
\begin{proof}
The feasible set of \eqref{eq:newformulation} is non-empty by \Cref{lemma:feasiblity}.
    Furthermore, the feasible set is bounded, as $ 0\leq M_t(i,j)\leq \mu_t(i)$ for all $t\in [\T]$ and $i,j\in [n]$, and closed, since it is defined by linear equality and inequality constraints. The objective function is finite and continuous, as it is smooth on the interior, and on the boundary we adopt the usual continuous extension
    $0\log\frac{0}{0}:=0$. Note that $M_t(i,j)\log(M_t(i,j)/\bar M(i,j))\to 0$ as $M_t(i,j),\bar M(i,j)\to 0$ since $M_t(i,j)\le \bar M(i,j)$.
    By Weierstrass theorem we can then conclude existence of a minimizer.
\end{proof}

\subsection{Duality}
In order to derive the corresponding dual problem the following lemma is useful.

\begin{lemma} 
\label{lemma:dualconstr}
    Let $\lambda \in \mR^\T$. The problem
    \[\inf_{x\in \mR_+^\T} \sum_{t=1}^\T \Big( x_t\log\frac{x_t}{\bar x} {\Big)}-\langle x,\lambda\rangle,
    \]
where $\bar x=\ett^T x$, has an optimal solution if and only if $\sum_t \exp(\lambda_t)\le 1$. In this case, the minimum value is 0 and the set of optimal solutions is given by
\begin{align*}
&\{0\}  && \mbox{ if }\quad     \sum_{t=1}^\T \exp(\lambda_t)< 1,\\
&\{x=\alpha \exp(\lambda)\mid {\alpha \geq 0}\} && \mbox{ if }\quad     \sum_{t=1}^\T \exp(\lambda_t)= 1.
\end{align*}
If $\sum_t \exp(\lambda_t)> 1$, then the objective value tends to $-\infty$ for $x^{(\alpha)}=\alpha \exp(\lambda)$  as $\alpha \to \infty$.
\end{lemma}

\begin{proof} 
Let $\beta:= \sum_{t}\exp(\lambda_t)$, and note that $\beta >0$. Define $F(x)$
\begin{align*}
           F(x):=&\sum_{t=1}^\T \Big(x_t\log\frac{x_t}{\bar x}\Big)- \langle x, \lambda\rangle = \sum_{t=1}^\T x_t\log\frac{x_t}{\bar x  \exp(\lambda_t)} \\
           =& \sum_{t=1}^\T x_t \log\frac{x_t}{\bar x e^{\lambda_t}/\beta}
\;-\;
\bar x \log\beta.
    \end{align*}
Next we will use the inequality $z \log(z/y) - z + y \geq 0$, where equality holds if and only if $z=y$, to show that the first term is non-negative. Thus we obtain
        \[ 
\sum_{t=1}^\T x_t \log\frac{x_t}{\bar x e^{\lambda_t}/\beta}
\ge
\sum_{t=1}^\T (x_t - \bar x e^{\lambda_t}/\beta)
= \bar x - \frac{\bar x}{\beta}\sum_{t=1}^\T e^{\lambda_t}
= 0, \]
from which it follows that $F(x) \geq - \bar x \log \beta$.
\begin{itemize}
    \item If $\beta<1$, then $F(x)>0$ for all $\bar x >0$, and $F(0)=0$. Hence $x=0$ is the unique minimizer.
    \item If $\beta=1$, then $F(x) \geq0$, with equality if and only if $x_t={\bar x e^{\lambda_t}}$ for all $t$. Therefore all minimizers are in the form $x=\alpha \exp(\lambda)$, for $\alpha \geq 0$.
    \item If $\beta>1$, then $-\bar x \log \beta < 0$, and along $x^{(\alpha)}=\alpha \exp(\lambda)$ we have $F(x^{(\alpha)})= -\alpha \beta \log\beta \rightarrow\ -\infty $ as $\alpha \rightarrow \infty$, so the problem is unbounded from below.
\end{itemize}
This completes the proof.
\end{proof}

We use the result of Lemma~\ref{lemma:dualconstr} to derive the Lagrange dual of \eqref{eq:newformulation}. We begin by expressing the corresponding Lagrangian
\begin{align*}
&\mathcal{L}(M,\lambda,\rho)
= \sum_{t=1}^\T \ccD(M_t \mid \bar M)
   + \sum_{t=1}^\T \lambda_t^T(\mu_t - M_t\ett) \\
&\; + \sum_{t=1}^\T \rho_t^T(\nu_t - M_t^T\ett) = \sum_{t=1}^\T (\lambda_t^T \mu_t + \rho_t^T \nu_t) \\
&\; + \sum_{i,j,t}\Big( M_t(i,j)\log\frac{M_t(i,j)}{\bar M(i,j)} - M_t(i,j)\bigl(\lambda_t(i)+\rho_t(j)\big)\Big),
\end{align*}
where $\lambda_t, \rho_t \in \mR^n$, for $t=1,\ldots,\T$, are the Lagrange multipliers.  Note that for a given element $(i,j)$, the problem of finding the minimizing $M_t(i,j)$, for $t=1,\ldots, \T$, is of the form considered in \Cref{lemma:dualconstr}. 
Therefore, the Lagrangian is bounded from below only if 
$\sum_{t} \exp\left( \lambda_t(i) + \rho_t(j) \right) \leq 1$, and the corresponding minimizers $M_t(i,j)$ are given by
\begin{equation}
\label{eq:optimalM}
    M^*_t(i,j)=\bar M^* (i,j)\exp(\lambda_t(i)+\rho_t(j)).
\end{equation}
In addition 
$M^*_t(i,j)=\bar M^* (i,j)=0$ for all $t$ if $\sum_t\exp(\lambda_t(i)+\rho_t(j))<1$.
The dual problem then becomes\\[-15pt]
\begin{subequations}\label{eq:dualproblem}
    \begin{align}
\label{eq:dual}
    \sup_{\substack{\lambda_t, \rho_t\in \mR^n\\ t\in [\T]}} \;\; &\sum_{t=1}^\T (\lambda_t^T \mu_t + \rho_t^T \nu_t) \\
    \suto\;   & \sum_{t=1}^\T \exp\left( \lambda_t(i) + \rho_t(j) \right) \leq 1, \; \forall i,j\in [n]. \label{eq:dualconstraint}
\end{align}
\end{subequations}
If $\mu_{t}>0$ and $\nu_{t}>0$ for $t\in [\T]$, then $M_t=\mu_t\nu_t^T/(\mu_t^T\ett)$ yield a feasible point with $M_t>0$ for all $t$, and hence $\bar M>0$. Therefore, by Slater's condition strong duality holds. If some entry of $\mu_{t}$ or $\nu_{t}$ is zero, then the corresponding row or column of $M_t$ is necessarily zero and can be removed without changing the problem. In that case, the dual optimum in \eqref{eq:dualproblem} is not attained, but the duality gap is zero, and it is approached in the limit as the dual variables associated with those entries tend to $-\infty$. For the uniqueness results in the following section we will assume that $\mu_{t}>0$ and $\nu_{t}>0$ for all $t\in [\T]$.

\subsection{Uniqueness}

Problem \eqref{eq:newformulation} is convex,  but not necessarily strictly convex, so its optimal solution may not be unique. We now characterize the structure of the optimal solution set of \eqref{eq:newformulation} and derive a necessary and sufficient condition for uniqueness. 

\begin{proposition}[Characterization of uniqueness]
\label{prop:uniqueness}
Let ($(M_t^*)_{t=1}^\T, \bar M^*$) be an optimal solution of~\eqref{eq:newformulation} and assume that $\mu_{t}>0$ and $\nu_{t}>0$ for all $t\in [\T]$. Let
 $u_t:=\exp(\lambda_t)$ and $v_t:=\exp(\rho_t)$ be the associated dual scalings, where $(\lambda_t,\rho_t)_{t=1}^\T$ are optimal dual variables, and denote by $\mathcal{I}=\{(i,j)\mid \sum_tu_t(i)v_t(j)<1\}$ the index pairs for which the dual constraints are inactive.
Then, the optimal solution is the unique solution if and only if 
the only solution $ X\in\mathbb{R}^{n\times n}$ to the system of equations 
\begin{subequations}\label{eq:setX}
\begin{align}
&(X\odot(u_t v_t^T))\,\ett = 0,\quad \: \:  t\in [\T], \label{eq:feasibilityX-a}\\
& (X\odot(u_t v_t^T))^T \ett = 0,\quad t \in [\T], \label{eq:feasibilityX-b}  \\
& X(i,j)=0 \text{ for } (i,j)\in \mathcal{I}, \label{eq:feasibilityX-c} \\
& X(i,j)\ge 0 \text{ for all } (i,j) \text{ where } \bar M^*(i,j)=0, \label{eq:feasibilityX-d}
\end{align}
\end{subequations}
is the trivial solution $X=0$.
 \end{proposition}

\begin{proof}
Exploiting the convexity of the optimal solution set, we consider optimal solutions of the form $M_t^*+\delta M_t$ to give conditions for uniqueness. Let \((\lambda,\rho)\) be optimal dual variables. By strong duality, any primal optimal solution of \eqref{eq:newformulation} minimizes the Lagrangian with \((\lambda,\rho)\) fixed. %
Fix an index pair $(i,j)$, and denote $m_t:=M_t(i,j)$, $\bar m=\sum_t m_t$. 
The $(i,j)$-th component of the Lagrangian of \eqref{eq:newformulation} reads
\[
\mathcal{L}_{ij}(m)
= \sum_{t=1}^\T \!m_t(\log(m_t/\bar m))
- \sum_{t=1}^\T (\lambda_t(i)+\rho_t(j))\,m_t.
\]
Minimizing $\mathcal{L}_{ij}$ with respect to $m$ yields the same problem as in Lemma~\ref{lemma:dualconstr}, 
with $\lambda_t$ replaced by $\lambda_t(i)+\rho_t(j)$. 
Hence, for the entries $(i,j)\in \mathcal{I}$ corresponding to an inactive dual constraint, the set of minimizers is $\{0\}$. Therefore, $M^*_t(i,j)=0$ is the unique solution for $(i,j)\in \mathcal{I}$.  

Consider now the entries $(i,j)\notin \mathcal{I}$, corresponding to active dual constraints, i.e., $(i,j)$ for which $\sum_tu_t(i)v_t(j)=1$. Then, the entire ray
\[
\big\{\alpha (u_t(i)v_t(j))_{t=1}^\T : \alpha\ge0\big\}
\]
gives cost-equivalent primal values. 
Collecting all entries together, an infinitesimal perturbation that preserves optimality in every active entry can then be written as
\[
\delta M_t = X \odot (u_t v_t^T), \qquad t \in [\T],
\]
for a matrix $X\in\mR^{n\times n}$ with $X_{ij}=0$ if $(i,j)\in \mathcal{I}$.
Feasibility with respect to the marginal constraints 
$M_t\ett=\mu_t$ and $M_t^\T \ett=\nu_t$
imposes
\[
(\delta M_t)\ett=0,
\qquad
(\delta M_t)^T\ett=0,
\qquad  t \in [\T].
\]
Thus, we obtain the linear feasibility system
\begin{equation*}
\begin{aligned}
(X\odot (u_t v_t^T))\ett &= 0, &  t \in [\T], \\
(X\odot (u_t v_t^T))^T\ett &= 0, &  t \in [\T]. 
\end{aligned}
\end{equation*}
Finally, the positivity constraint $M \geq 0$ imposes that, if $\bar M^*(i,j)=0$, it must also hold  $X(i,j) \geq 0$.
\end{proof}

The condition that $X=0$ is the only solution to \eqref{eq:setX} can be expressed directly in terms of the dual scaling vectors.

\begin{corollary}[Dual characterization]
\label{cor:uniq-dual}
Let $(u_t,v_t)_{t=1}^\T$ be the dual scaling vectors at the optimum, with $u_t=\exp (\lambda_t)$ and $v_t=\exp(\rho_t)$. 
Then a sufficient condition for a solution $M^*$ of problem~\eqref{eq:newformulation} to be unique is that 
\[
\Span\{v_t\}_{t=1}^\T=\mR^n
\quad\text{or}\quad
\Span\{u_t\}_{t=1}^\T=\mR^n .
\]
Furthermore, if $\bar M^*>0$, the condition is also necessary.
\end{corollary}

\begin{proof}
To prove sufficiency, note that if either $\{v_t\}_t$ or $\{u_t\}_t$ spans $\mR^n$, then from
\eqref{eq:feasibilityX-a} or \eqref{eq:feasibilityX-b} it follows that 
$X(i,:)=0$ or $X(:,j)=0$ for all $i,j$, hence $X=0$, proving uniqueness.
To prove necessity, assume that $\bar M^*(i,j)>0$ for all $(i,j)$.  Now, assume that both families have deficient span, i.e. there exist nonzero vectors 
\[
a \in \Span\{u_t\}_{t=1}^\T{}^\perp,
\qquad
b \in \Span\{v_t\}_{t=1}^\T{}^\perp.
\]
Define $X\in\mR^{n\times n}$ by $X(i,j)=a_i b_j$. 
For each $t$,
\begin{equation*}
    \begin{split}
        \big[(X \odot (u_t v_t^T))\ett\big]_i
&= u_t(i)\,a_i\,\langle b, v_t\rangle = 0,
\\
\big[(X \odot (u_t v_t^T))^T\ett\big]_j
&= v_t(j)\,b_j\,\langle a, u_t\rangle = 0,
    \end{split}
\end{equation*}
since $a\perp u_t$ and $b\perp v_t$ for all $t$. 
Thus $X\neq0$ satisfies \eqref{eq:feasibilityX-a} and \eqref{eq:feasibilityX-b}. Under the positivity assumption $\bar M^* > 0$, also \eqref{eq:feasibilityX-c} and  \eqref{eq:feasibilityX-d} are satisfied. 
Therefore, we have constructed a nonzero solution of \eqref{eq:setX}. By \Cref{prop:uniqueness}, the solution is then not unique. Hence, if the solution is unique, at least one of the families $\{u_t\}_{t=1}^\T$ or $\{v_t\}_{t=1}^\T$ must span $\mR^n$, proving necessity.

\end{proof}

The dual characterization provides a compact and practically verifiable condition, with uniqueness holding if the dual scaling vectors $u_t$ or $v_t$ span $\mathbb{R}^n$. We note that for this condition to be satisfied it must hold $\T \geq n$. 
The following result presents a primal interpretation of \Cref{cor:uniq-dual}.
\begin{corollary}[Primal characterization]
\label{cor:uniq-primal}
Let $M^*$ be an optimal solution to problem~\eqref{eq:newformulation}, and assume that $\bar M^* >0$. Then $M^*$ is the unique optimal solution if and only if,  for every $i$ (equivalently, for every $j$),
\[
\Span\{M_t^*(i,:)\}_{t=1}^\T = \mR^n
\quad\text{or}\quad
\Span\{M_t^*(:,j)\}_{t=1}^\T = \mR^n,
\]
i.e., each of the rows (or the columns) of $(M^*_t)_{t=1}^\T$ spans the whole space. 
\end{corollary}

\begin{proof}
By optimality condition \eqref{eq:optimalM}, $M_t^*(i,j) = \bar M^*(i,j)\,u_t(i)\,v_t(j).$
Fix $i$. For each $t$,
\[
M_t^*(i,:)
= u_t(i)\,\big(\bar M^*(i,:)\odot v_t^T\big),
\]
so all the row vectors $\{M_t^*(i,:)\}_t$ are obtained from $\{v_t\}_t$ by diagonal scaling with the fixed positive vector $\bar M^*(i,:)$.
Since scaling by a fixed diagonal matrix with strictly positive entries preserves span, we get
\[
\Span\{M_t^*(i,:)\}_{t=1}^\T = \Span\{v_t\}_{t=1}^\T.
\]
A symmetric argument with the columns gives
\[
\Span\{M_t^*(:,j)\}_{t=1}^\T = \Span\{u_t\}_{t=1}^\T,
\]
and the uniqueness result follows from \Cref{cor:uniq-dual}.
\end{proof}

\begin{remark}
Linear independence of the dual scalings $(u_t,v_t)$ reflects the marginals $(\mu_t,\nu_t)$. 
If two observations are proportional, $(\mu_{t_1},\nu_{t_1})=c\,(\mu_{t_2},\nu_{t_2})$, then $u_{t_1}\propto u_{t_2}$ and $v_{t_1}\propto v_{t_2}$, so the second observation adds no information. 
The precise link between the rank of $(u_t,v_t)$ and that of $(\mu_t,\nu_t)$ remains an open question.
\end{remark}

\section{Numerical method and convergence analysis}
\label{sec:nm}
The main idea for solving the optimization problem \eqref{eq:newformulation} is to iteratively update the transport plans and transition probability matrix. 
Updating the transport plans consists of computing the solution of $\T$ entropy regularized optimal transport problems, which can be solved efficiently (and in parallel) with Sinkhorn iterations \cite{cuturi2013sinkhorn,peyre2019computational}. The prior estimate is then updated as the sum of the optimal mass transport matrices, and the procedure iterated. 
This is detailed in Algorithm~\ref{alg:entropy}.

\begin{algorithm}
\caption{Proximal Iterative Scheme for \eqref{eq:newformulation}}\label{alg:entropy}
\begin{algorithmic}
\State $M_t \gets \mu_t \nu_t^T {/ \mu_t ^T \ett}, \ X \gets \sum_{t=1}^\T M_t$
    \While{change in $X$ above tolerance} 
        \State $M_t \gets \argmin_{M_t \in \mR^{n \times n}_+} \  \mathcal{D}(M_t \mid X)$
\State \hspace{3em} $\suto \;\;  \, M_t \ett = \mu_t,\; M_t^T\ett = \nu_t,\; \ t \in [\T]$
        \State $X \gets \sum_{t=1}^\T M_t$
\EndWhile
  \State $A \gets \diag{(X\mathbf{1})}^{-1}\, X$
\end{algorithmic}
\end{algorithm}

A way to view this algorithm, and to show its convergence, is to consider it as an instance of the entropic proximal method \cite{teboulle1992entropic}, with Kullback-Leibler divergence~\eqref{eq:kl} as the penalization. 
The method allows to minimize a closed, convex, and proper function  \( f : \mathbb{R}^m \rightarrow (-\infty, +\infty] \), by generating the sequence \( \{x^k\} \), starting from an initial strictly positive point \( x^0 > 0 \), according to the update rule
\begin{equation}
\label{eq:proximal}
    x^{k+1} = \arg\min_{x \in \mR^m_+} \left\{ f(x) + \epsilon_k \, \ccD(x |x^{k}) \right\},
\end{equation}
for a sequence of positive parameters \( \{\epsilon_k\} \). The next proposition shows convergence of \Cref{alg:entropy}, and its connection to the proximal iterations \eqref{eq:proximal}.

\begin{proposition}
    Under \Cref{as:nonzero}, \Cref{alg:entropy} converges to  an optimal solution $\bar M^*$ of \eqref{eq:newformulation}.
    \end{proposition}
\begin{proof}
Consider the function
\begin{equation}
\begin{aligned}
\label{eq:minf}
 f(X) = \inf_{\substack{ M_t\in \mR^{n\times n}_+\\ t \in [\T]}}\quad & \ \sum_{t=1}^{\T} \ccD( M_t \,|\,X) \\
\suto \;  &  M_t\ett =\mu_t & \mbox{for }  t\in [\T] \\
  &M_t^T\ett=\nu_t & \mbox{for }  t\in [\T] \\
  & X= \sum_{t=1}^\T M_t.
\end{aligned}
\end{equation}
The function \( f \) corresponds to problem \eqref{eq:newformulation}, with the prior $X$ being considered as the variable.
First, observe that $f$ is closed, convex and proper.
Under \Cref{as:nonzero}, we have seen (\Cref{lemma:feasiblity}) that there exists a strictly positive feasible point, so that $\dom f \cap \mR^{n \times n}_{>0} \neq \emptyset $. In addition, \Cref{prop:existence} guarantees that the optimal solution set is non-empty, thus $\{X : f(X)=\inf_{X \in \mR^{n \times n}_{+}} f \}  \neq \emptyset  $. With the parameter choice  $\epsilon_k\equiv1$, the entropic proximal point algorithm converges according to \cite[Theorem~4.3]{teboulle1997convergence}, with the objective gap decaying as \(O(1/K)\), denoting with \(K\) the number of iterations.
The iterations of \eqref{eq:proximal}, with $f$ defined as \eqref{eq:minf}, become
\begin{equation*}
\begin{aligned}
\argmin_{X  \in \mR_+^{n \times n}}\;\inf_{\substack{ M_t\in \mR^{n\times n}_+\\ t \in [\T]}}\;
   &\sum_{t=1}^{\T} \ccD(M_t \,|\, X) + \ccD(X \,|\, X^k)\\[3pt]
\suto\;
   & M_t\ett = \mu_t, \quad M_t^T\ett = \nu_t, \; \mbox{ for } t\in[\T]\\
   & X = \sum_{t=1}^{\T} M_t.
\end{aligned}
\end{equation*}
Using the constraint \(X=\sum_{t=1}^{\T} M_t\), the objective can be equivalently expressed as $\sum_{t=1}^{\T} \ccD( M_t \,|\,X^k)$. 
Thus, the variable $X$ is not part of the objective function, and will be determined by the optimal mass transport matrices $M_t$ via the constraint $X=\sum_t M_t$. The updates become
\begin{equation*}
\begin{aligned}
    M^* &\gets 
    \begin{cases}
        \displaystyle \argmin_{M_t \in \mR^{n \times n}_+} \sum_{t=1}^\T \mathcal{D}(M_t \,|\, X^k), \\[6pt]
        \quad \text{s.t. } M_t \mathbf{1} = \mu_t,\; M_t^T \mathbf{1} = \nu_t,\; t \in [\T],
    \end{cases} \\[3pt]
    X^* &\gets \sum_{t=1}^\T M_t^*,
\end{aligned}
\end{equation*}
which correspond to \Cref{alg:entropy}.
\end{proof}
\begin{remark}
    As it is stated, \Cref{alg:entropy} requires to solve, at each iteration, $\T$ bi-marginal entropy regularized OT problems. This computational cost can be greatly alleviated by performing only a few sweeps of the Sinkhorn iterations in the inner loop,  often just one. In particular, if each minimization in \eqref{eq:proximal} is performed with an accuracy $\eta_k$, and  $\sum_{k=1}^\infty \epsilon_k \eta_k < \infty$, then the inexact iterations converge to an optimal solution \cite[Theorem~4.3]{teboulle1997convergence}.  
     Recent works have specialized this result for optimal transport and some of its related formulations \cite{yang2022bregman, chen2024inexact, xie2020fast}, giving sufficient descent conditions for convergence.
     \end{remark}

\section{Applications}
\label{sec:application}

A typical application of this problem is when the marginal distributions are derived from independent Markov chain trajectories. In particular, if $\{X^i_t\}_{i=1}^N$ are $N$ independent trajectories of a Markov chain with $n$ states $\{1,\ldots,n\}$ and Markov transition matrix $A$, define 
\begin{equation}
\label{eq:trajectory}
    \mu_t(k) = \frac{1}{N}\sum_{i=1}^N \ett_{X^i_t=k},\quad \text{for}\; k \in [n],
\end{equation}
and $\nu_t = \mu_{t+1}$. For finite but large $N$, we have 
\begin{equation} \label{eq:noise}
    \nu_t = A^T \mu_t + \text{``noise''},
\end{equation}
where the ``noise" part becomes zero when $N\to \infty$.
In this section, we numerically investigate the estimation of the Markov transition matrix $A$ from such measurements.  
First, in Section~\ref{sec:indep-obs}, we consider independent measurements, where multiple Markov transitions of length 2 are considered. Then, in Section~\ref{sec:seq-obs}, we consider sequential observations where each pair is of the form $(\mu_t, \mu_{t+1})$ for $t\in [\T]$. In Section~\ref{sec:discussion} we discuss the identifiability issue in connection with the excitation of the states. 
Finally, in Section~\ref{sec:comparison} we compare the proposed methodology with a constrained least-squares estimator.
In all experiments, the outer iterations stopping tolerance has been set to $10^{-6}$, while the inner Sinkhorn updates were performed inexactly, with one step per outer iteration.

\subsection{Independent observations}\label{sec:indep-obs}
To illustrate \Cref{alg:entropy}, we consider the matrix $A$, arising from the estimation from flow cytometry data in \cite{swaminathan2014identification},
\begin{equation} \label{eq:fixedA}
  A =
\begin{bmatrix}
0.48 & 0.50 & 0 & 0.02 & 0 \\
0.33 & 0.27 & 0 & 0.40 & 0 \\
0 & 0 & 0 & 0.54 & 0.46 \\
0.26 & 0 & 0.45 & 0.29 & 0 \\
0 & 0 & 0.51 & 0 & 0.49
\end{bmatrix}.  
\end{equation}
We simulate a scenario where the Markov chain \(A\) produces \(\T\) marginal pairs \((\mu_t, \nu_t)\), corresponding to two consecutive trajectory observations as in \eqref{eq:trajectory}.
In particular, for each $t$, we  sample
 $\mu_t$ from a random uniform distribution, and then $\nu_t$ is constructed according to \eqref{eq:trajectory}.  
 Thus, each pair represents distinct realizations of the same process with independent initial condition, ensuring excitation of all the states of the system.
Increasing finite values of \(N\) are used to model varying noise levels: as $N\rightarrow \infty$ we reach the transition $\nu=A^T \mu$. 
\begin{figure}
    \centering
    \includegraphics[width=\linewidth]{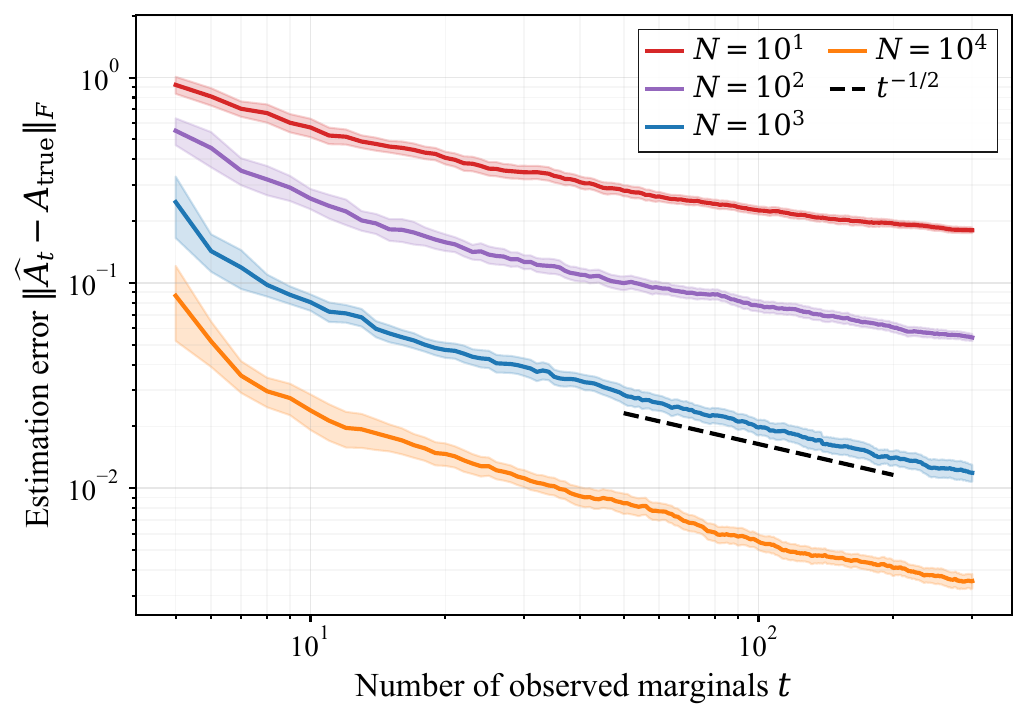}
    \caption{Numerical results for the application of Algorithm~\ref{alg:entropy} on the independent observation case of Section~\ref{sec:indep-obs}. The plot shows estimation error as a function of observed marginals (in \textit{log-log} scale), for four values of the particle number $N$.}
    \label{fig:placeholder}
\end{figure}
We test the recovery of the true transition matrix \(A\) as \(\T\) increases, using the Frobenius norm between \(A\) and its estimate~\eqref{eq:optimalA}; see~\Cref{fig:placeholder}.
The system has $5$ states, and an increasing number of observation pairs $\T \in \{5,\ldots,300\}$ is considered. For each value of $N$, a total of $K=30$ simulations with different observation pairs have been averaged. The shaded area is the 95\% confidence interval of the mean error across the repeats.

After an initial transient due to early noise, we observe that, when $N$ is large enough for the system to approach the regime \eqref{eq:noise}, the error converges at a rate close to \(t^{-1/2}\), typical of consistent estimators with increasing i.i.d.\ data~\cite[Ch.~5]{van2000asymptotic}.%

\subsection{Sequential observation}\label{sec:seq-obs}
We now consider a scenario in which the number of particles \(N\) is small, and the marginal distributions correspond to consecutive observations of particle trajectories as in \eqref{eq:trajectory}, so that each observation pair is in the form $(\mu_t,\mu_{t+1})$, for $t\in [\T]$.
\begin{figure}
    \centering
    \includegraphics[width=\linewidth]{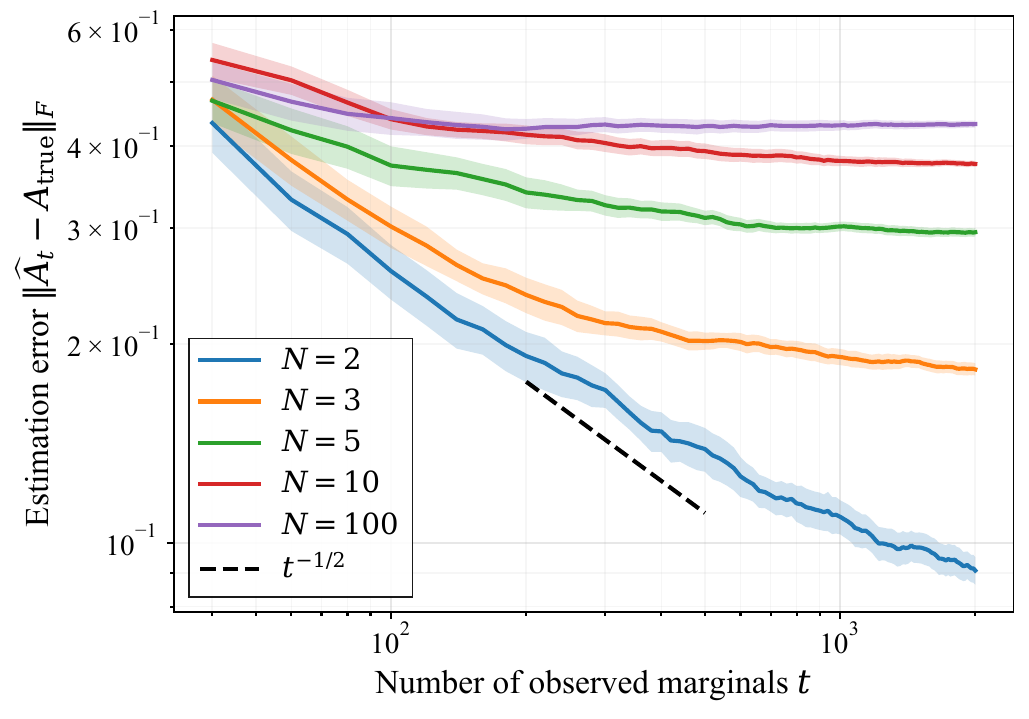}
    \caption{Numerical results for the application of Algorithm~\ref{alg:entropy} on the sequential observation case of Section~\ref{sec:seq-obs}, for $A$ given by \eqref{eq:fixedA}. The plot shows estimation error as a function of observed marginals (in \textit{log-log} scale), for five  values of the particle number $N$.}
    \label{fig:placeholder2}
\end{figure}
The results are shown in \Cref{fig:placeholder2}. 
We consider a number of marginals in the range \(\{40,\ldots,2000\}\). The best performance is achieved with \(N=2\) particles: the reduced number of particles introduces higher stochasticity, making the transitions closer to independent. On the other hand, with few observations, a small number of particles such as $N=2$ may not explore all states, whereas larger values of $N$ yield a solution even for smaller $\T$, but they do not provide enough excitation to the system, and their estimation error reaches a nearly constant level. This phenomenon will be further discussed in the next subsection.

\subsection{Discussion on identifiability}
\label{sec:discussion}
If the matrix $A$ is irreducible and aperiodic, we observe exponential convergence of the marginal observations \eqref{eq:trajectory} to the unique stationary distribution $\pi$ of $A$ \cite{levin2017markov},
\begin{equation*}
   d(t):= \sup_{\mu_0}\norm{(A^t)^T \mu_0 -\pi}_{TV} \leq C \alpha^t,
\end{equation*}
for $C>0$, while the coefficient $\alpha \in (0,1)$ is related to the second largest (in magnitude) eigenvalue of $A$, and $\norm{\mu-\nu}_{TV}:=\frac{1}{2}\sum_i |\mu(i)-\nu(i)|$ is the total variation norm. It is then possible to define the mixing time
$$
t_{mix}(\epsilon):=\min\{ t: d(t)\leq \epsilon \} ,
$$
at which the distance to stationarity is $\epsilon$-small.
As a consequence, approaching the mixing time, any two Markov chains $A$ and $B$ that share the same stationary distribution $\pi$ will lead to very similar trajectories, $(A^t)^T \mu_0 \approx (B^t)^T \mu_0 \approx \pi$.
Thus after a few time steps, the observations reduce to (noisy) draws from the stationary distribution $\pi$, and it will not be possible to distinguish among the elements of the set
\begin{equation*}
    \mathcal{A}(\pi):=\{A \in \mR^{n \times n}_+: A \ett = \ett, \ A^T \pi = \pi\}.
\end{equation*}
The convex optimization problem \eqref{eq:problem2} will then pick a (possibly unique, if the conditions in \Cref{prop:uniqueness} are satisfied) optimal solution in  $\mathcal{A}(\pi)$, corresponding to the best fit to near-stationary, noisy marginals.
This behavior is observed in \Cref{fig:placeholder2}: as $N$ increases, the system receives insufficient excitation, and the estimation error stabilizes to an almost constant level. In particular, for $N=100$, the estimate of $A$ at $\T=2000$, averaged across the repeats, is
\begin{equation}
    \label{eq:Ahat} \hat{A}=
    \begin{bmatrix}
0.39 & 0.37 & 0.07 & 0.11 & 0.06 \\
0.30 & 0.23 & 0.08 & 0.32 & 0.07 \\
0.08 & 0.05 & 0.10 & 0.41 & 0.35 \\
0.24 & 0.06 & 0.34 & 0.27 & 0.09 \\
0.08 & 0.04 & 0.39 & 0.12 & 0.37 \\
\end{bmatrix}.
\end{equation}
Denoting with $\pi, \hat{\pi}$, the stationary distributions of $A, \hat{A}$, respectively, we observe that $\norm{\pi-\hat{\pi}}_{TV}\approx0.001$. Therefore, the method selected an element  $\hat{A} \in \mathcal{A}(\pi)$ (up to numerical approximation),
whose entries are all positive and well separated from zero.
This discrepancy likely originates from the zero entries of \eqref{eq:fixedA}, which are penalized by the maximum-entropy formulation that favors a more uniform distribution of mass, as in \eqref{eq:Ahat}. 
We test this hypothesis by repeating the experiment with a randomly generated, positive $A$. The results are reported in \Cref{fig:placeholder3}. We observe better convergence properties than in \Cref{fig:placeholder2}, although slower than the $-{\textstyle\frac{1}{2}}$ rate observed with independent data.
If a sparsity pattern is known a priori, it can be incorporated directly into \eqref{eq:primalfirstform} by restricting the support of $A$ and of the transport plans, i.e., by imposing $A_{ij}=0$ and $M_t(i,j)=0$ on the forbidden entries.
\begin{figure}
    \centering
    \includegraphics[width=\linewidth]{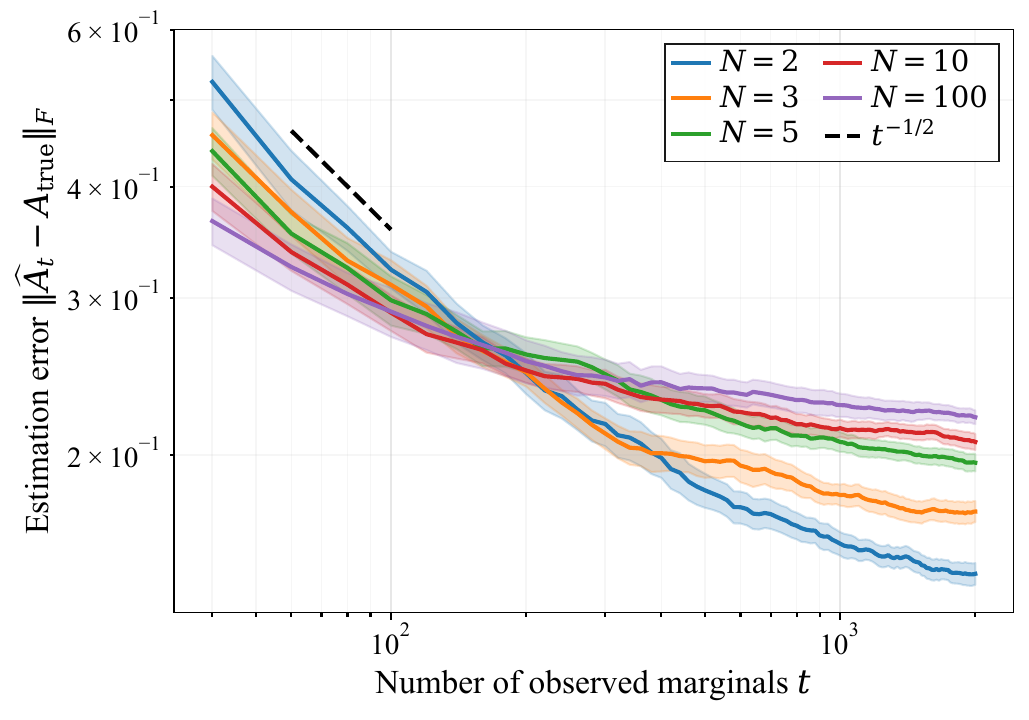}
    \caption{Numerical results for the application of Algorithm~\ref{alg:entropy} on the sequential observation case of Section~\ref{sec:seq-obs}, with random positive $A$. The plot shows estimation error as a function of observed marginals (in \textit{log-log} scale), for five  values of the particle number $N$.}
    \label{fig:placeholder3}
\end{figure}

To correctly identify the underlying transition probability matrix $A$, it is then necessary that the observations contain sufficient excitation. %
As $\norm{\mu_t -\pi}$ decays exponentially with time (in the noise-free case $N\to \infty$), an effective way of improving identifiability is to collect several independent episodes with different initial data, analogous to the setup in \Cref{sec:indep-obs}.
Furthermore, when the number of particles $N$ is finite, it also offers a trade-off between noise and exploration of $A$. Low values of $N$ lead to higher stochasticity, which enhances exploration and makes the marginal distributions more informative, at the cost of a less accurate measurement of $A$.

\subsection{Comparison with a least-squares baseline}
\label{sec:comparison}
We complement the previous experiments with a comparison against a constrained least-squares estimator, proposed in \cite{lee1970estimating}. Given  $\T$ aggregate observations $(\mu_t,\nu_t)$, we consider the following optimization problem
\begin{equation} \label{eq:lsq-baseline}
\begin{aligned}
    \min_{A \in \mR^{n\times n}} \; \ & 
    \sum_{t=1}^{\T} \| \nu_t - A^T \mu_t \|_2^2 \\
    \text{subject to } \ & 
    A \ett = \ett, \qquad A \geq 0.
\end{aligned}
\end{equation}
Problem~\eqref{eq:lsq-baseline} is a convex quadratic program and is solved using CVXPY~\cite{diamond2016cvxpy} with an off-the-shelf QP solver. 
In this comparison, we use the same random dense positive transition matrix $A$ as was used in Section~\ref{sec:discussion} for which the error estimates were presented in \Cref{fig:placeholder3}. We compare the two estimators with sequential observations (i.e., $\nu_t=\mu_{t+1}$) and with independently generated pairs of marginals, reporting the average over 30 observation scenarios.
\begin{figure}[t]
    \centering
    \begin{subfigure}{0.48\textwidth}
        \centering
        \includegraphics[width=\textwidth]{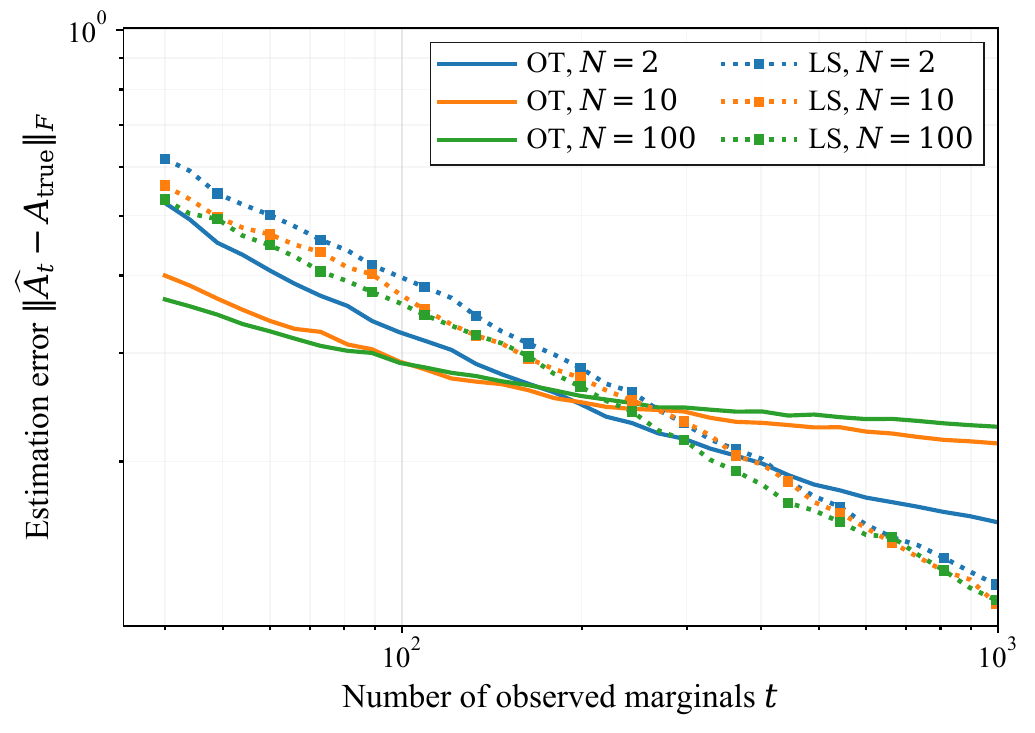}
        \caption{Sequential observations.}
        \label{fig:baseline-dense-sequential}
    \end{subfigure}
    \vfill    \vspace{15pt}
    \begin{subfigure}{0.48\textwidth}
        \centering
        \includegraphics[width=\textwidth]{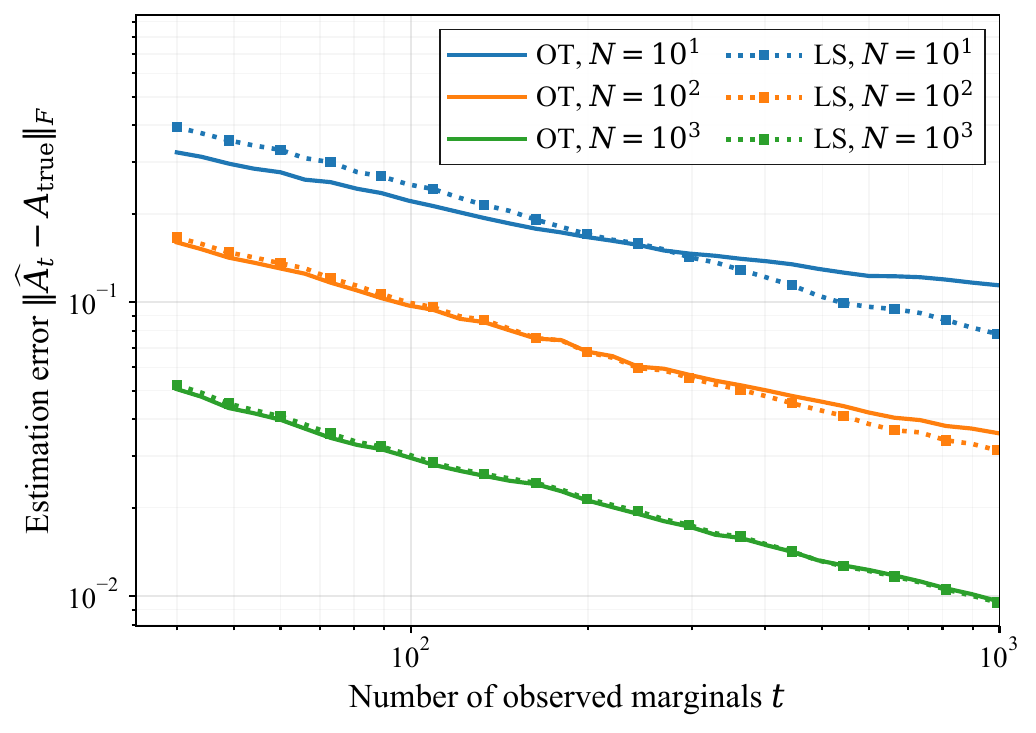}
        \caption{Independent observations.}
        \label{fig:baseline-dense-independent}
    \end{subfigure}
    \caption{
    Comparison between the proposed estimator and the constrained least-squares estimator for a randomly generated dense transition matrix. The plots show estimation error as a function of observed marginal pairs (in \textit{log-log} scale), for three  values of the particle number $N$.
    }
    \label{fig:baseline-dense-comparison}
\end{figure}
The results are reported in Figure~\ref{fig:baseline-dense-comparison}. In both settings, the proposed methodology gives smaller reconstruction error for small values of \(\T\). As the number of observed marginal pairs increases, the least-squares estimator eventually attains a lower reconstruction error. The difference is particularly prominent in the \textit{sequential observation} scenario, while with independent observations, as $N$ increases, the two estimators have a closer performance.

This comparison illustrates the different behavior induced by the two formulations.
While the least-squares baseline directly fits the empirical transition relations \(\nu_t \approx A^T\mu_t\), which are obtained from the particle dynamics up to sampling fluctuations, the proposed formulation instead estimates \(A\) through an entropy-based inverse optimal transport problem, and simultaneously reconstructs also the optimal mass transportation matrices \(M_t\).

We also evaluate the computational scaling of the proposed OT estimator, by reporting in \Cref{fig:runtime} the runtime for the  independent observation setting with $N=100$ particles per marginal and an increasing number $\T$ of marginals, for different stopping tolerances. The same dense transition matrix as in the previous experiments has been used. The reported runtimes are averaged over 30 observation scenarios. The simulation was performed on a laptop with an Intel Core i7-1165G7 CPU and 16 GB of RAM, using Python 3.9.
\Cref{fig:runtime} shows that the problem can be solved within seconds on a standard laptop computer with high precision even when the number of marginals is large.
As expected, the runtime increases with the number of observed marginals $\T$, reflecting the increasing size of the optimization problem. 
Stricter stopping tolerances lead to longer runtimes, although the dependence on $\T$ appears similar across tolerances.

\begin{figure}
    \centering
    \includegraphics[width=0.97\linewidth]{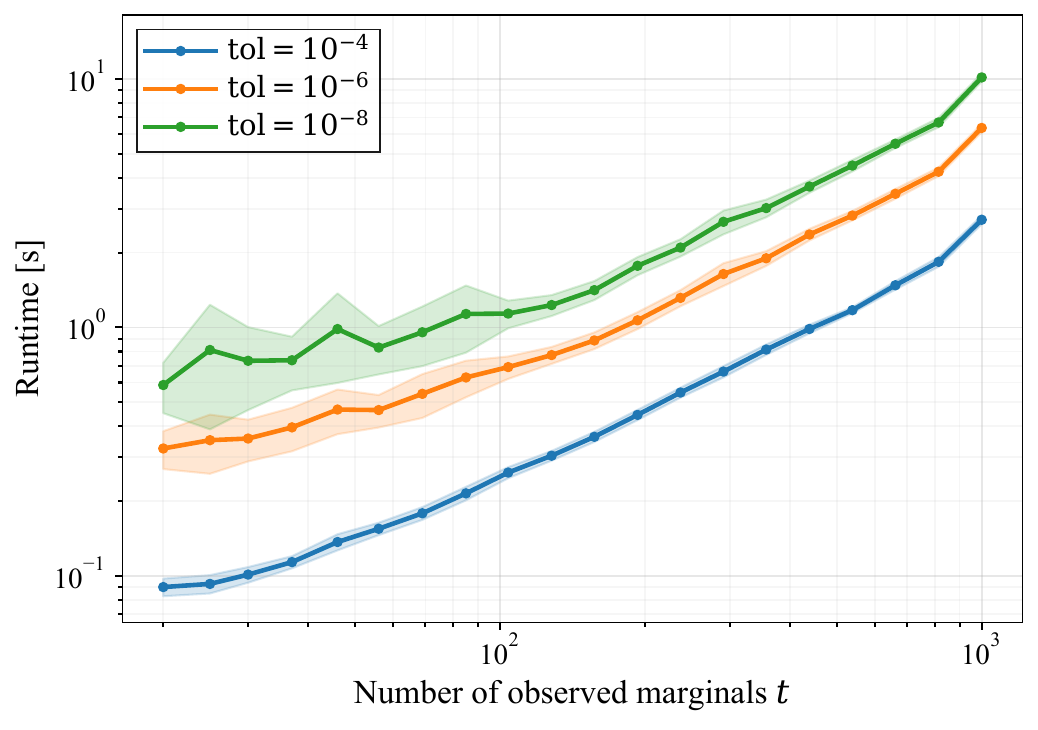}
    \caption{Runtime of the OT estimator in the independent-observation setting with $N=100$, averaged over 30 scenarios.}
    \label{fig:runtime}
\end{figure}

\section{Conclusions}
\label{sec:conc}
In this work, we addressed the problem of estimating a Markov chain from aggregate observations by formulating it as an entropy-regularized optimal transport problem, where both the cost function and the transport matrices are treated as optimization variables under marginal constraints given by the data. We proposed an entropic proximal algorithm that alternately updates the transport plans and the transition matrix, showing with numerical experiments that the method successfully retrieves the true transition probabilities, provided sufficient excitation of the system.

Future research directions include a theoretical investigation of identifiability, as well as providing a  proof of consistency and establishing the convergence rate observed empirically under independent observations. Another line of work concerns deriving conditions for uniqueness of the optimal solution based on the properties of the observed data rather than on the solution itself. Possible generalizations include partial model identification, where either prior information on the transition structure is incorporated or certain entries of the transition matrix are constrained or fixed, as well as investigating extensions to hidden Markov models, cf.~\cite{singh2022learning}.

\section*{Declaration of generative AI and AI-assisted technologies in the manuscript preparation process}
 During the preparation of this work the authors used ChatGPT to assist with text refinement. After using this tool/service, the authors reviewed and edited the content as needed and take full responsibility for the content of the published article.

\bibliographystyle{abbrv}

\bibliography{./references}

\end{document}